\documentclass{article}
\usepackage{amsthm}
\usepackage{amsmath}
\usepackage{amssymb}
\usepackage{setspace}
\usepackage{mathtools}
\usepackage{verbatim}
\usepackage{a4wide}
\usepackage{csquotes}
\usepackage[hidelinks]{hyperref}
\usepackage{thm-restate}
\usepackage{cleveref}
\usepackage{enumitem}
\setlist[itemize]{leftmargin=*} 
\setlist[enumerate]{leftmargin=*}
\usepackage{framed}
\usepackage{floatrow}
\usepackage[T1]{fontenc}
\usepackage{bbm}
\usepackage[bottom]{footmisc}

\usepackage{mathdots}
\usepackage{xcolor}
\usepackage{diagbox}
\usepackage{colortbl}

\usepackage{graphicx}
\usepackage{subcaption}

\theoremstyle{plain}

\newtheorem{theorem}{Theorem}[section]
\newtheorem{claim}[theorem]{Claim}

\newtheorem{lemma}[theorem]{Lemma}

\newtheorem{problem}[theorem]{Problem}
\theoremstyle{definition}

\newtheorem{defn}[theorem]{Definition}
\newtheorem*{defn*}{Definition}
\newtheorem*{lem*}{Lemma}

\expandafter\def\expandafter\normalsize\expandafter{%
    \normalsize
    \setlength\abovedisplayskip{4pt}
    \setlength\belowdisplayskip{4pt}
    \setlength\abovedisplayshortskip{4pt}
    \setlength\belowdisplayshortskip{4pt}
}

\usepackage[square,sort,comma,numbers]{natbib}
\def\abhi#1{}
\def\domagoj#1{}

\let\abhi=\abhiOpt 
\let\domagoj=\domagojOpt 

\newcommand{\calP}{\mathcal{P}}

\def\eps {\varepsilon}
\DeclareMathOperator{\E}{\mathbb{E}}

\newcommand{\ex}{\mathrm{ex}}

\title{The extremal number of cycles with all diagonals}
\author{Domagoj Brada\v{c}\thanks{Department of Mathematics, ETH, Z\"urich, Switzerland. Research supported in part by SNSF grant 200021\_196965. Emails:~\textbf{\{domagoj.bradac, abhishek.methuku, benjamin.sudakov\}@math.ethz.ch}.} \and Abhishek Methuku\footnotemark[1] \and Benny Sudakov\footnotemark[1]}
\date{}
\begin{document}

\maketitle

\begin{abstract}
In 1975, Erd\H{o}s asked the following natural question: What is the maximum number of edges that an $n$-vertex graph can have without containing a cycle with all diagonals? Erd\H{o}s observed that the upper bound $O(n^{5/3})$ holds since the complete bipartite graph $K_{3,3}$ can be viewed as a cycle of length six with all diagonals.

In this paper, we resolve this old problem. We prove that there exists a constant $C$ such that every $n$-vertex with $Cn^{3/2}$ edges contains a cycle with all diagonals. Since any cycle with all diagonals contains cycles of length four, this bound is best possible using well-known constructions of graphs without a four-cycle based on finite geometry. 
Among other ideas, our proof involves a novel lemma about finding an \emph{almost-spanning} robust expander which might be of independent interest. 

\end{abstract}

\section{Introduction}

One of the most central topics in extremal combinatorics, the Tur\'an problem, asks how many edges are needed in an $n$-vertex graph to guarantee the existence of a certain substructure.
In this paper we resolve the problem of finding a cycle with all diagonals and develop a novel method which may have other applications.

The study of extremal problems for cycles goes back to the influential result of Mantel~\cite{mantel1907vraagstuk} from 1907 on triangle-free graphs, and the work of Erd\H{o}s~\cite{erdos1944problem} connecting Sidon sets with the extremal problem for four-cycles. Since then Tur\'an-type problems have been extensively studied, which lead to the development of many tools and methods, with applications in many other fields such as discrete geometry and information theory. Formally, the \emph{extremal number} of a graph $F$, introduced by Tur\'an in 1941 and denoted by $\ex(n,F)$, is the maximum possible number of edges in an $n$-vertex graph not containing $F$ as a subgraph. The celebrated Erd\H{o}s–Stone–Simonovits theorem~\cite{stone1946structure, erdos1966limit} states $\ex(n, F) = (1 - \frac{1}{\chi(F)-1} + o(1)) \binom{n}{2}.$ This result implies that if $F$ is not bipartite, $\ex(n, F) = \Theta(n^2)$. However, if $F$ is bipartite, it only shows that $\ex(n, F) = o(n^2)$. In fact, there are still only a few bipartite graphs $F$ for which the order of magnitude of $\ex(n, F)$ is known. A fundamental result of Bondy and Simonovits~\cite{bondy1974cycles} shows that $\ex(n,C_{2 \ell}) = O(n^{1+1/\ell})$, where $C_{2 \ell}$ denotes a cycle of length $2 \ell$. Matching lower bound constructions are only known for cycles of length four, six and ten~\cite{brown1966graphs, wenger1991extremal, erdHos1966problem}. In particular, it is known that $\ex(n,C_{4}) = \Theta(n^{3/2})$. For the complete bipartite graph $K_{s,t}$, a well-known result of K\H{o}v\'{a}ri, S\'{o}s and Tur\'{a}n~\cite{kHovari1954problem} from 1954 shows the upper bound $\ex(n,K_{s,t}) = O(n^{2-1/s})$ holds for any $t \ge s$, and matching lower bound constructions are known when $s=t=2,3$  and when $t$ is large enough compared to $s$, see for example~\cite{kollar1996norm,alon1999norm,bukh2021extremal}. Even the extremal numbers of simple graphs such as the cube and $K_{4,4}$ remain elusive.  

Questions about finding cycles in graphs with given properties have been extensively studied, see for instance~\cite{corradi1963maximal, egawa1996vertex, verstraete2000arithmetic,liu2023solution}, and the nice survey~\cite{verstraete2016extremal}. In particular, there has been a lot of research on problems concerning finding cycles with chords. In 1975, Erd\H{o}s~\cite{erdHos1975problems} conjectured that there is a constant $c$ such that any $n$-vertex graph with $cn$ edges contains two cycles $C_1, C_2$ where the edges of $C_{2}$ are chords of the cycle $C_1$. Bollob\'as~\cite{MR539938} later proved this conjecture. Extending his result, Chen, Erd\H{o}s and Staton~\cite{chen1996proof} proved that there is a constant $c_k$ such that any $n$-vertex graph with $c_k n$ edges contains $k$ cycles $C_1,\ldots, C_k$, such that the edges of $C_{i+1}$ are chords of the cycle $C_i$. Fern{\'a}ndez, Kim, Kim and Liu~\cite{fernandez2022nested} strengthened the aforementioned result of Bollob\'as by finding cycles $C_1, C_2$ where the edges of $C_{2}$ do not cross each other and form chords of the cycle $C_1$. Recently, improving an old result of Chen, Erd\H{o}s and Staton~\cite{chen1996proof}, Dragani\'c, Methuku, Munh\'a Correia and Sudakov~\cite{draganic2023cycles} showed that any $n$-vertex graph with $\Omega(n \log^8 n)$ edges contains a cycle with at least as many chords as it has vertices. 

In this paper we are interested in finding a cycle containing all chords joining vertices at maximum distance on the cycle. Such chords are called \emph{diagonals}. Since the extremal number of any non-bipartite graph is $\Theta(n^2)$, it is more interesting to find an even cycle with all diagonals. Formally, a cycle of length $2 \ell$ with all diagonals, denoted $C^{\textrm{dia}}_{2 \ell}$, is a graph with the vertex set $V(C^{\textrm{dia}}_{2 \ell}) \coloneqq \{x_1, \dots, x_{2 \ell}\}$ and the edge set $E(C^{\textrm{dia}}_{2 \ell}) \coloneqq E(C_{2 \ell}) \cup \{x_i x_{i+ \ell} \mid i = 1, \dots, \ell \}$ (see Figure~\ref{fig:c10}).

\begin{figure}[h]
    \centering
\includegraphics{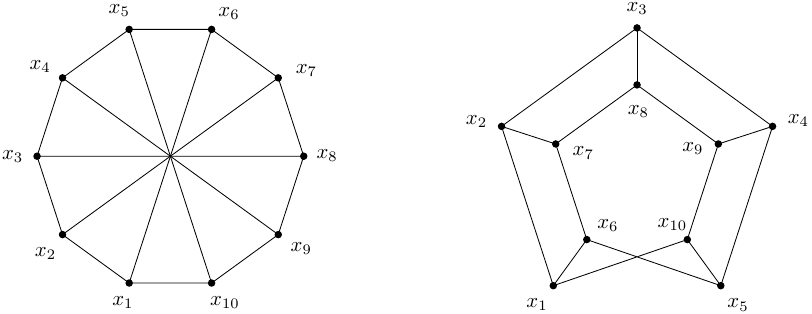}
    \caption{Two different ways of drawing the graph $C^{\textrm{dia}}_{10}$.}
    \label{fig:c10}
\end{figure}

In 1975, Erd\H{o}s~\cite{erdos1975some} asked the following natural extremal question: What is the maximum number of edges that an $n$-vertex graph can have if does not contain a cycle with all diagonals? In 1995, this problem was reiterated in a well-known paper of Pyber, R\"{o}dl, and Szemer\'{e}di~\cite{pyber1995dense} on the Erd\H{o}s-Sauer problem~\cite{erdos1975some, janzer2022resolution}. Erd\H{o}s~\cite{erdos1975some} observed that the upper bound $\ex(n, K_{3,3}) = O(n^{5/3})$ holds since the complete bipartite graph $K_{3,3}$ can be viewed as the cycle of length six with all diagonals.  In this paper, we resolve this old problem. 

\begin{theorem} \label{thm:main}
    There exists a constant $C > 0$ such that every $n$-vertex graph with at least $Cn^{3/2}$ edges contains a cycle with all diagonals. 
\end{theorem}

Observe that every cycle with all diagonals contains cycles of length four. Hence, by the aforementioned lower bound $\ex(n, C_4) = \Omega(n^{3/2})$, the bound in Theorem~\ref{thm:main} is tight up to a constant factor. 

Note that $C^{\textrm{dia}}_{2 \ell}$ is bipartite if and only if $\ell$ is odd. Hence, in this paper, we focus on finding a cycle $C^{\textrm{dia}}_{2 \ell}$ where $\ell = 2k+1$ for some integer $k \ge 1$. It is convenient to view the graph $C^{\textrm{dia}}_{2 \ell} = C^{\textrm{dia}}_{4k+2}$ as an ``odd cycle of four-cycles" of length $2k+1$. See the drawing on the right in Figure~\ref{fig:c10} for an example (and see Claim~\ref{claim:cycle} for a formal argument). Note that one of the four-cycles in the figure contains a ``twist'', otherwise the resulting graph is not bipartite.  

\textbf{Related results.} Let us now mention some closely related results,  and briefly explain how our approach differs from the existing methods and the difficulties we encounter.

The \emph{$\ell$-prism} $C_{\ell}^{\square}:=C_{\ell}\square K_{2}$ is the Cartesian product of the cycle $C_{\ell}$ of length $\ell$ and an edge. In other words, $C_{\ell}^{\square}$ consists of two vertex disjoint $\ell$-cycles and a matching joining the corresponding vertices on these two cycles. For instance, $C_{4}^{\square}$ is the notorious cube graph, and the best known upper bound is $\ex(n, C_{4}^{\square}) = O(n^{8/5})$~\cite{pinchasi2005graphs}. Observe that the \emph{$\ell$-prism} $C_{\ell}^{\square}$ and $C^{\textrm{dia}}_{2 \ell}$ are both ``cycles of four-cycles" on $2 \ell$ vertices but the crucial difference is that in $C^{\textrm{dia}}_{2 \ell}$ one of the four-cycles contains a twist, while in $C_{\ell}^{\square}$ none of the four-cycles have a twist. This is the reason why, if $\ell$ is odd, $C^{\textrm{dia}}_{2 \ell}$ is bipartite but $C_{\ell}^{\square}$ is not bipartite. He, Li and Feng~\cite{he2023extremal} studied the odd prisms and determined $\ex(n, C_{2k+1}^{\square})$ for $k \ge 1$ and large $n$. In that case $\ex(n, C_{2k+1}^{\square}) = \Theta(n^2)$ since $C_{2k+1}^{\square}$ is not bipartite. Gao, Janzer, Liu and Xu~\cite{gao2023extremal}  studied the even prisms and showed that $\ex(n, C_{2\ell}^{\square}) = \Theta(n^{3/2})$ for every $\ell \ge 4$. Let us compare the problem of finding $C_{2\ell}^{\square}$ with that of finding $C^{\textrm{dia}}_{2 \ell}$. As discussed above, $C_{2\ell}^{\square}$ corresponds to an ``even cycle of four-cycles" but $C^{\textrm{dia}}_{2 \ell}$ corresponds to an ``odd cycle of four-cycles". Since forcing an odd cycle in a graph is significantly more difficult than forcing an even cycle (because the extremal number of an odd cycle is $\ex(n, C_{2k+1}) = \Omega(n^2)$), we need a new approach for finding $C^{\textrm{dia}}_{2 \ell}$, which we discuss in detail in Section~\ref{sec:overview}.

Next, let us mention a problem about finding (topological) cycles in pure simplicial complexes with an interesting parallel to the problem of finding $C^{\textrm{dia}}_{2 \ell}$ and $C_{2\ell}^{\square}$. A $3$-uniform hypergraph $\mathcal H$ corresponds to the $2$-dimensional pure simplicial complex obtained by taking the downset generated by the set of edges of $\mathcal H$. A tight cycle of even length is homeomorphic to the cylinder $S^1 \times B^1$, while a tight cycle of odd length is homeomorphic to the M\"{o}bius strip (see Figure 1 in~\cite{tomon2022robust}). 
In~\cite{kupavskii2022extremal}, Kupavskii, Polyanskii, Tomon, and Zakharov proved that every $3$-uniform hypergraph with $n$ vertices and at least $n^{2+\alpha}$ edges contains a triangulation of the cylinder. However, their proof does not work for finding a triangulation of the M\"{o}bius strip. This is because they reduce this problem to a problem about finding rainbow cycles in certain properly edge colored graphs where an even cycle corresponds to a cylinder, but an odd cycle corresponds to a M\"{o}bius strip. As the extremal number of odd cycles is $\Omega(n^2)$, their method completely breaks down for finding a triangulation of the M\"{o}bius strip. Note that, interestingly, a triangulation of the M\"{o}bius strip is analogous to a cycle with all diagonals (as they both contain a twist) and a triangulation of the cylinder is analogous to a prism $C_{2\ell}^{\square}$. 
Overcoming this issue, Tomon~\cite{tomon2022robust} proved that every $3$-uniform hypergraph with $n$ vertices and at least $n^{2+\alpha}$ edges contains a triangulation of the M\"{o}bius strip by passing to a subhypergraph which is an expander (where subsets of $1$-dimensional faces expand via $2$-dimensional faces). Along these lines, a natural approach for finding a cycle with all diagonals in a graph $G$ would be to pass to an expander where four-cycles expand via edges. However, this approach immediately fails for our problem because it is completely possible that such an expander is bipartite in our case, hence it is impossible to find an odd cycle in it. This issue is explained in more detail in Section~\ref{sec:overview}. To overcome this fundamental obstacle, we develop a novel lemma that finds an almost-spanning expander.

Let us now discuss the extremal numbers of two more natural graphs containing four-cycles: the \emph{grid} and the \emph{blow-up of a cycle}. For a positive integer $t$, the grid $F_{t,t}:=P_{t}\square P_{t}$ is the Cartesian product of two vertex-disjoint paths $P_{t}$. Since $F_{t,t}$ contains four-cycles, we have 
$\ex(n, F_{t,t}) = \Omega(n^{3/2})$. Brada\v{c}, Janzer, Sudakov and Tomon~\cite{bradavc2023turan} showed that $\ex(n, F_{t,t}) = O(n^{3/2})$. Gao, Janzer, Liu and Xu~\cite{gao2023extremal} later gave a very simple proof of this bound. Roughly speaking, it is easier to find the grid compared to the graphs $C^{\textrm{dia}}_{2 \ell}$ or $C_{2\ell}^{\square}$ because the four-cycles in the grid are not arranged in a cyclic manner. Finally, let us consider blow-ups of cycles. The $2$-blowup of a cycle is the graph obtained by replacing each vertex of the cycle with an independent set of size $2$ and each edge of the cycle by a four-cycle. Jiang and Newman~\cite{jiang2017small} asked the following question. What is the maximum number of edges that an $n$-vertex graph can have without containing the $2$-blowup of a cycle? Janzer~\cite{janzer2023rainbow} proved that an upper bound of $O(n^{3/2}(\log n)^{7/2})$ holds, and it is known that a lower bound of $\Omega(n^{3/2})$ holds. It is an interesting question to decide whether the logarithmic factor in the upper bound can be removed.

\textbf{Notation.} For a set $S$, we denote its complement by $\overline{S}$. For a graph $G$, let $v(G)$ denote the number of vertices of $G$, let $e(G)$ denote the number of edges of $G$ and let $e_G(X, Y)$ denote the number of edges of $G$ with one endpoint in $X$ and another endpoint in $Y$. For a set $S \subseteq V(G)$, let $e_G(S)$ denote the number of edges of $G$ spanned by $S$, and let $N_{G}(S) \coloneqq \{y \not \in S \mid xy \in E(G) \text{ for some } x \in S\}$ be the neighbourhood of $S$ in $G$. Sometimes we write $N(S)$ and $e(S)$ instead of $N_{G}(S)$ and $e_G(S)$ respectively if the graph $G$ is clear from context.  For any two distinct vertices $u, v \in V(G)$, let $d_G(u, v)$ denote the size of the common neighbourhood of $u$ and $v$ in $G$. Let $\delta(G)$ and $\Delta(G)$ denote the minimum degree and maximum degree of $G$ respectively. For a set of vertices $V' \subseteq V(G)$, let $G \setminus V'$ denote the graph obtained by removing the vertices of $V'$ from $G$, and for a set of edges $E' \subseteq E(G)$, let $G - E'$ denote the graph obtained by removing the edges in $E'$ from $G$. 

\textbf{Organisation of the paper.} In Section~\ref{sec:overview}, we give a detailed overview of the proof of Theorem~\ref{thm:main}. In Section~\ref{sec:expansion}, we introduce the notion of edge-expansion we use and prove a preliminary lemma. In Section~\ref{sec:proof}, we prove Theorem~\ref{thm:main} which is split into three well-defined subsections. Our main lemma about finding an almost-spanning expander is proved in Section~\ref{subsec:mainlemma}. Finally, we give some concluding remarks and open problems in Section~\ref{sec:concluding}.

\section{Overview of the proof}
\label{sec:overview}
Even though our proof is relatively short, in this section we give a detailed overview of the proof of Theorem~\ref{thm:main} highlighting certain novel aspects of the proof.  Suppose $G$ is an $n$-vertex graph with $Cn^{3/2}$ edges where $C$ is a large enough constant, and we may assume $G$ is bipartite using standard arguments. Our aim is to show that $G$ contains a cycle with all diagonals. 

As noted in the introduction, since a cycle of length $2 \ell$ with all diagonals, $C^{\textrm{dia}}_{2 \ell}$, is bipartite if and only if $\ell$ is odd, we assume $\ell = 2k+1$ for some integer $k \ge 1$, and we view the graph $C^{\textrm{dia}}_{2 \ell} = C^{\textrm{dia}}_{4k+2}$ as an ``odd cycle of four-cycles'' of length $2k+1$ as shown in Figure~\ref{fig:c10} (see Claim~\ref{claim:cycle} for a formal argument).

This suggests the following natural strategy. Consider an auxiliary graph $\Gamma_0$ whose vertex set is the set of edges of $G$ and $xy, uv \in E(G)$ are adjacent in $\Gamma_0$ if $xyuv$ is a four-cycle in $G$. We call $\Gamma_0$ the \emph{$C_4$-graph} of $G$. By a standard supersaturation result, $G$ has at least $\Omega(C^4 n^2)$ four-cycles. Hence, 
$\Gamma_0$ has $N = Cn^{3/2}$ vertices and at least $\Omega(C^4 n^2) = \Omega(N^{4/3})$ edges. Observe that a cycle of length $2k+1$ in $\Gamma_0$ corresponds to a copy of $C^{\textrm{dia}}_{4k+2}$ in $G$ if the vertices of the cycle correspond to pairwise disjoint edges in $G$; otherwise, we obtain a \emph{degenerate} copy of $C^{\textrm{dia}}_{4k+2}$ in $G$. Thus, our goal is to embed an \emph{odd} cycle in $\Gamma_0$ which corresponds to a non-degenerate copy of $C^{\textrm{dia}}_{4k+2}$ in $G$; call such an odd cycle in $\Gamma_0$, \emph{proper}. To find such an embedding, it is very helpful if we can find paths connecting pairs of vertices in $\Gamma_0$ in a robust manner i.e., avoiding a given set of already embedded vertices in $G$.  Using expansion is a standard way of achieving this, so it is natural to take a subgraph $\Gamma \subseteq \Gamma_0$ which is an expander, and try to embed an odd cycle in $\Gamma$.   

We, however, encounter a fundamental difficulty with the above strategy. Note that since the extremal number of the four-cycle is less than $n^{3/2} = N / C = v(\Gamma_0)/C,$ there is no independent set of size $N / C$ in $\Gamma_0$ i.e., $\Gamma_0$ is far from being bipartite in some sense. So there are odd cycles in $\Gamma_0$ but when we pass to the expander subgraph $\Gamma \subseteq \Gamma_0$, it is quite possible that $v(\Gamma)$ is much smaller than $v(\Gamma_0)$ and $\Gamma$ can even be bipartite, in which case, there is no hope for finding an odd cycle in $\Gamma$ (let alone, a proper odd cycle).

To overcome the above difficulty, we prove the following lemma (Lemma~\ref{lem:almostspanninginformal}). We believe this lemma is of independent interest and similar ideas may have other applications (see Lemma~\ref{lem:almost_spanning_expander}
for its formal statement). Note that since every graph contains an expander, 
without loss of generality, we may assume the host graph $G$ is an expander. In our proof we actually need that $G$ has very strong edge expansion; for this, we adopt the notion of $\alpha$-maximality from Tomon~\cite{tomon2022robust} (see Definition~\ref{def:alphamaximal}). Moreover, we also need that the maximum degree of $G$ is close to its average degree and to guarantee this, we remove a carefully chosen set of vertices from a $1/2$-maximal subgraph: this affects the expansion of very small sets but all large enough sets (of size up to $v(G)/2$) in $G$ still expand extremely well (see Lemma~\ref{lem:expandersubgraph}).

\begin{lemma}\label{lem:almostspanninginformal}
If $G$ is an edge-expander, then its $C_4$-graph $\Gamma_0$ contains an almost-spanning (robust) expander $\Gamma$.    
\end{lemma}

More precisely, Lemma~\ref{lem:almostspanninginformal} shows that we obtain an excellent robust (vertex) expander $\Gamma$ by removing only a small proportion of vertices and edges from $\Gamma_0$ i.e., $v(\Gamma) \ge (1-\varepsilon) v(\Gamma_0)$ for some small $\varepsilon > 0$. This ensures that the independence number of $\Gamma$ is at most $v(\Gamma_0)/C \le v(\Gamma)/2C$, so it is now feasible to find a proper odd cycle in $\Gamma$. 

Before outlining the proof of Lemma~\ref{lem:almostspanninginformal}, let us sketch how to complete the proof assuming the lemma holds. Since Lemma~\ref{lem:almostspanninginformal} ensures $v(\Gamma) \ge (1-\varepsilon) v(\Gamma_0) = (1-\varepsilon) N \ge Cn^{3/2}/2$,  a simple supersaturation argument implies that $\Gamma$ is \emph{locally-dense} in the following sense: For any $U \subseteq V(\Gamma)$ with $|U| \ge N / 100,$ we have $e(\Gamma[U]) \ge \Omega(C^4 n^2).$
Since $\Gamma$ is a robust expander, starting at any given vertex in $\Gamma$, almost all other vertices of $\Gamma$ can be reached by a short \emph{proper} path. Moreover, the set of these reachable vertices contains many edges of $\Gamma$  because $\Gamma$ is locally-dense. Using this, we obtain that most pairs of vertices $x, y$ in $\Gamma$ are connected by a short odd proper path $P_o(x,y)$ as well as a short even proper path $P_e(x,y)$ in a robust manner i.e., avoiding a small set of vertices in $G$ but it is still possible that $P_o(x,y)$ and $P_e(x,y)$ have a common vertex in $G$ (and to establish this robustness property, we actually show that the edges of the expander $\Gamma$ correspond to four cycles in $G$ in which any opposite pair of vertices has small codegree, see the last paragraph of the proof overview for more details). Combining this with ideas from Letzter~\cite{letzter2023hypergraphs}, we obtain a pair $x, y$ in $\Gamma$ such that there are linearly many proper paths $\{Q_i(x,y)\}_i$ connecting $x$ and $y$, $P_o(x,y)$ and $P_e(x,y)$ are defined, and moreover, every vertex of $G$ in $P_o(x,y)$ and $P_e(x,y)$ does not appear on too many of the paths $\{Q_i(x,y)\}_i$. This implies that there is a proper path $Q_j(x,y)$ such that either $Q_j(x,y) \cup P_o(x,y)$ or $Q_j(x,y) \cup P_e(x,y)$ is the desired proper odd cycle in $\Gamma$.

Let us now sketch the proof of Lemma~\ref{lem:almostspanninginformal} (ignoring robustness for simplicity).
We first argue that all sets of size at least $\varepsilon N$ and at most $0.99 N$ in $\Gamma_0$ expand well. Suppose for a contradiction that there is a set $E_b \subseteq V(\Gamma_0)$ with $\varepsilon N \le |E_b| \le 0.99 N$ that does not expand i.e., $N_{\Gamma_0}(E_b) < \delta|E_b|$ for some $\delta > 0$. Let $E_p \coloneqq N_{\Gamma_0}(E_b)$ and let $E_r \coloneqq  V(\Gamma_0) \setminus (E_b \cup E_p)$. Note that $E_b, E_p, E_r$ are vertex sets in $\Gamma_0$, so they are edge sets in $G$. For convenience, call the edges of $G$ in $E_b, E_p, E_r$, \emph{blue}, \emph{purple} and \emph{red} respectively. We now partition the vertices of $G$ into four sets $P, R, B, U$ as follows: Roughly speaking, a vertex $v \in V(G)$ is in $P$ if at least half of the edges incident to $v$ are purple; $v$ is in $R$ if most of the edges incident to $v$ are red; $v$ is in $B$ if most of the edges incident to $v$ are blue; $v$ is in $U$ if there are sufficiently many blue edges as well as red edges incident to $v$. Now, we claim that $B$ does not (edge) expand very well in $G$. Indeed, the number of purple edges in $G$ is small (because we assumed $|E_p| < \delta|E_b|$), and there are very few blue and red edges between $B$ and $R$ (by the definition of $B$ and $R$), thus $e_G(B, R)$ is very small. The number of edges incident to $P$ is at most four times the number of purple edges, which is assumed to be small. This implies $e_G(B, P)$ is small. Moreover, crucially, the size of $U$ must be very small, because each vertex of $U$ provides many blue-red paths (by definition) and if $U$ is large, we obtain a \emph{red-red-blue-blue} four-cycle (i.e., a four-cycle of the form $xuyv$ where $xu, xv$ are red and $uy, vy$ are blue). However, such a four-cycle in $G$ corresponds to an edge in $\Gamma_0$ between $E_b$ and $E_r$ which is impossible. This implies that $e_G(B, U)$ is also small. In total, we obtain that $e_G(B, \overline{B}) = e_G(B, R) + e_G(B, P) + e_G(B, U)$ is very small, i.e., $B$ does not expand well in $G$, as claimed. On the other hand we know that all large enough sets (of size at most $v(G)/2$) in $G$ expand very well, so either $|B|$ or $|\overline{B}|$ is very small, which in turn implies that either $|E_b|$ or $|E_r|$ is very small, a contradiction. Hence, all sets of size at least $\varepsilon N$ and at most $0.99 N$ in $\Gamma_0$ expand well, as desired. Now, removing a maximal non-expanding set among the sets of size at most $0.99 N$ in $\Gamma_0$, we obtain a (vertex) expander $\Gamma$. Since such a maximal set must have size at most $\varepsilon N$, $\Gamma$ is an almost-spanning expander of $\Gamma_0$, as required by Lemma~\ref{lem:almostspanninginformal}. 

Let us highlight that we actually need that sets in $\Gamma$ vertex-expand robustly i.e., while also avoiding a given set $B \subseteq V(G)$ of size, say, $|B| = \Omega(\log n)$. This is not straightforward to achieve. If we simply remove all vertices 
in $\Gamma$ corresponding to the edges of $G$ incident to $B$, we may lose more than $\Theta(\sqrt{n} \log n)$ vertices from $\Gamma$ which makes it impossible for small sets in $\Gamma$ to expand robustly (in which case, starting at a vertex, we cannot robustly reach other vertices and our proof fails). 
To overcome this difficulty, we prove that $\Gamma$ actually satisfies two crucial properties: $\mathrm{(i)}$ $\Gamma$ is a robust expander in a very strong sense (that ties vertex and edge expansion together) as follows: every set $S$ (which is not too large) in $\Gamma$ expands well even after the removal of any small set $F$ of edges in $\Gamma$, see Definition \ref{def:robust-expansion}. $\mathrm{(ii)}$ Every edge in $\Gamma$ corresponds to a four-cycle $xyzw$ in $G$ for which both $d_G(x,z)$ and $d_G(y,w)$ are small enough.
Then, $\mathrm{(ii)}$ ensures that the number of edges in $\Gamma$ incident to a set $S$ which have an endpoint in $B$ are small enough, so that the notion of robustness in $\mathrm{(i)}$ can be used to show that $S$ expands well while also avoiding $B$.

\section{Expansion}
\label{sec:expansion}

In this section we show that every graph contains a subgraph with maximum degree close to its average degree and with very strong edge-expansion (see Lemma~\ref{lem:expandersubgraph}). For this we need the following notion of $\alpha$-maximal graphs. The idea of taking a subgraph maximizing some density function dates back to the work of Koml\'{o}s and Szemer\'edi \cite{komlos1996topological} and the definition we are going to use has recently appeared in the work of Tomon~\cite{tomon2022robust}.

\begin{defn}[$\alpha$-maximal graphs]
\label{def:alphamaximal}
Let $\alpha > 0$. A graph $G$ is $\alpha$-maximal if for every subgraph $H$ of $G$, we have $$ \frac{e(H)}{v(H)^{1+\alpha}} \le \frac{e(G)}{v(G)^{1+\alpha}}.$$ In other words, if $e(G) = C v(G)^{1+\alpha}$, then $e(H) \le C v(H)^{1+\alpha}$ for every subgraph $H$ of $G$. 
\end{defn}

In \cite{tomon2022robust} it is shown that $\alpha$-maximal graphs are excellent edge-expanders. This is made precise in the following simple lemma. We include its short proof here for completeness. 

\begin{lemma}[Lemma 2.5 in~\cite{tomon2022robust}] 
\label{edgeexpander}
Let $0 < \alpha <1$, let $G$ be an $\alpha$-maximal graph on $n$ vertices, and let $e(G) = C n^{1+\alpha}$. Then $G$ satisfies the following properties.
\begin{itemize}
    \item[$\mathrm{(i)}$] If $G$ is non-empty, then $C \ge 1/4$. 
    \item[$\mathrm{(ii)}$] Let $S \subset V(G)$, and let $|S| \le n/2$. Then $e(S, V(G) \setminus S) \ge |S|  \frac{\alpha C n^{\alpha}}{4}$.
\end{itemize}
\end{lemma}
\begin{proof}
Since $G$ is non-empty, it contains an edge. Let $H$ be the subgraph of $G$ consisting of this edge. Since $G$ is $\alpha$-maximal, we have $e(H) \le C v(H)^{1+\alpha}$, so $1 \le C 2^{1+\alpha}$ i.e., $C \ge 1/4$, showing (i) holds. 

Now we show (ii). We have $e(S, \overline{S}) = e(G) - e(S) - e(\overline{S}) \ge C(|S| + |\overline{S}|)^{1+ \alpha} - C|S|^{1+ \alpha} - C|\overline{S}|^{1+ \alpha}$ since $G$ is $\alpha$-maximal. Moreover, since $\left (1 + \frac{|S|}{|\overline{S}|} \right)^{1+ \alpha} \ge 1 + (1+\alpha) \frac{|S|}{|\overline{S}|}$, $|\overline{S}| \ge n/2$, and $|S| \le |\overline{S}|$ we have
\begin{equation*}
C(|S| + |\overline{S}|)^{1+ \alpha} - C|\overline{S}|^{1+ \alpha} - C|S|^{1+ \alpha} \ge C( (1+ \alpha)|S||\overline{S}|^{\alpha} - |S|^{1+ \alpha}) = C |S| ((1+ \alpha)|\overline{S}|^{\alpha} - |S|^{\alpha}) \ge \alpha C |S| (n/2)^{\alpha}.
\end{equation*}
Noting that $(n/2)^{\alpha} \ge n^{\alpha}/4$, this proves (ii).
\end{proof}

Note that every graph $G$ contains an $\alpha$-maximal subgraph since a subgraph $H$ of $G$ maximising the quantity $e(H)/v(H)^{1+\alpha}$ (over all subgraphs) is $\alpha$-maximal. Thus Lemma~\ref{edgeexpander} shows that every graph contains a subgraph which is an excellent edge-expander. In our proof, we will also need that the maximum degree of this subgraph is not too large. In the next lemma we show that every graph contains a subgraph $H$ whose maximum degree is close to its average degree, and all large enough sets in the subgraph $H$ expand very well. For convenience, we only state the lemma for $\alpha = 1/2$ but it holds more generally, and may be of independent interest.

\begin{lemma}
\label{lem:expandersubgraph}
    Let $C_1, D$ be constants satisfying $1000 \le D \le C_1 / 1000$ and let $G$ be an $n$-vertex graph with $C_1 n^{3/2}$ edges. Then, $G$ contains an $m$-vertex subgraph $H$ satisfying the following properties for some $C > 0$.
    
    \begin{enumerate}[label=$\mathrm{(P\arabic*)}$]
        \item $e(H) = C m^{3/2},$ where $C \ge C_1 / 2,$ \label{prop:edges}
        \item $\Delta(H) \le CD  m^{1/2},$ \label{prop:max-degree}
        \item for every $S \subseteq V(H)$, $e(S) \le 2C |S|^{3/2},$ \label{prop:no-dense-set} 
        \item for every $S \subseteq V(H), \frac{1000 m}{D} \le |S| \le \frac{m}{2},$ $e_H(S, V(H) \setminus S) \ge |S|  \frac{C m^{1/2}}{16}$. \label{prop:edge-expansion} 
    \end{enumerate}    
\end{lemma}
\begin{proof}
    Let $H_0$ be a $1/2$-maximal subgraph of $G$ and let $v(H_0) = m$ and $e(H_0) = 2C_0 m^{3/2}$. Then $2 C_0 \ge C_1$ (since $H_0$ is $1/2$-maximal). Let $U$ be the set of vertices in $H_0$ with degree at least $C_0 D m^{1/2}$ and let $t$ denote the number of edges in $H_0$ incident to $U$. Note that $|U| \le \frac{2t}{C_0 D m^{1/2}} \eqqcolon s < \frac{m}{2},$ where in the last inequality we used $D \ge 1000$. Let $W$ be a random subset of $V(H_0) \setminus U$ of size $s$. Then, $\E[H_0[U \cup W]] \ge ts / m.$ Indeed, any edge of $H_0$ with both endpoints in $U$ is also in $H_0[U \cup W]$ and an edge of $H_0$ with a single endpoint in $U$ appears in $H_0[U \cup W]$ with probability at least $s / m.$ So there is a set $Q$ of size at most $2s < m$ vertices spanning at least $ts / m$ edges. Since $H_0$ is $1/2$-maximal, we have $ts/m \le e(H_0[Q]) \le 2C_0 (2s)^{3/2}$.
    Rearranging and using the definition of $s$ yields $t \le \frac{64 C_0 m^{3/2}}{D}.$
    Remove from $H_0$ all the $t$ edges incident to $U$, and call this new graph $H$ (which is still on $m$ vertices) and denote $e(H) = Cm^{3/2} = 2C_0 m^{3/2} - t \ge 2C_0 m^{3/2} - \frac{64 C_0 m^{3/2}}{D}$. Then, since $D \ge 1000$, we have $2 C_0 \ge C \ge C_0 \ge C_1 / 2,$ so \ref{prop:edges} holds.
    
    Since we removed all edges incident to vertices of degree at least $C_0 D m^{1/2}$ when defining $H$, we have $\Delta(H) \le C_0 D m^{1/2} \le C D m^{1/2}$, so \ref{prop:max-degree} holds. Moreover, since $H_0$ is $1/2$-maximal, we have $e_{H_0}(S) \le 2 C_0 |S|^{3/2}$, so $e_{H}(S) \le 2 C_0 |S|^{3/2} \le 2 C |S|^{3/2}$, so \ref{prop:no-dense-set} holds.  
    
    By Lemma~\ref{edgeexpander} (ii), for every set $S$ satisfying $|S| \le \frac{m}{2},$ we have $e_{H_0}(S, V(H_0) \setminus S) \ge |S| \frac{C_0 m^{1/2}}{4}$, so $e_{H}(S, V(H) \setminus S) \ge  |S| \frac{C_0 m^{1/2}}{4}  - t \ge |S| \frac{C_0 m^{1/2}}{4} - \frac{64 C_0 }{D} m^{3/2}$. Moreover, if $|S| \ge \frac{1000 m}{D}$,  we have $|S| \frac{C_0 m^{1/2}}{4} - \frac{64 C_0 }{D} m^{3/2} \ge |S| \frac{C_0 m^{1/2}}{8} \ge |S| \frac{C m^{1/2}}{16} .$ Thus \ref{prop:edge-expansion} holds.
\end{proof}

\section{Proof of Theorem~\ref{thm:main}}
\label{sec:proof}

In this section, we prove Theorem~\ref{thm:main}.
Throughout the proof, we fix small positive constants $\delta, \eta, \eps$ and large constants $C_1, D$ such that $$1 \gg \eps \gg \eta \gg \delta \gg 1/D \gg 1/C_1 > 0.$$

 Take a graph $G_0$ with $e(G_0) \ge 4C_1 v(G_0)^{3/2}.$ Our goal is to show that $G_0$ contains a cycle with all diagonals. It is convenient for us to work with a suitable subgraph of $G_0$ instead of $G_0$. Let $G_1$ be a bipartite subgraph of $G_0$ with $e(G_1) \ge e(G_0)/2 \ge 2C_1 v(G_0)^{3/2}$, and let $G$ be the (expander) subgraph obtained by applying Lemma~\ref{lem:expandersubgraph} to $G_1$.  For convenience, denote $v(G) = n$. Then \ref{prop:edges} and \ref{prop:max-degree} of Lemma~\ref{lem:expandersubgraph} ensure that $e(G) = Cn^{3/2}$ and $\Delta(G) \le CD n^{1/2}$ for some $C \ge C_1$. It suffices to show that $G$ contains a cycle with all diagonals.  We introduce an auxiliary graph as follows.

\begin{defn}[$C_4$-graph]
For a bipartite graph $G$ with the bipartition $\{X,Y\}$, its $C_4$-graph $\Gamma(G)$ is defined as a graph with the vertex set $V(\Gamma(G)) = E(G)$ and $xy, x'y' \in V(\Gamma(G))$ are adjacent in $\Gamma(G)$ if and only if $xyx'y'$ is a four-cycle in $G,$ where $x, x' \in X, y, y' \in Y.$ Note that every $4$-cycle in $G$ contributes $2$ edges to $\Gamma(G)$.
\end{defn}

 Let $\Gamma_0 \coloneqq \Gamma_0(G)$ be the $C_4$-graph of $G$.
A four-cycle $xyx'y'$ in $G$ is called \emph{thick} if $\max\{d_G(x, x'), d_G(y, y')\} \ge 10 \sqrt{CD} n^{1/4},$ otherwise it is called \emph{thin}. 
Let $F_T$ be the set of edges of $\Gamma_0$ corresponding to the thick four-cycles in $G$. By the following simple claim we will be able to assume that there are not many thick four-cycles in $G$. A similar idea is used by Gao, Janzer, Liu and Xu in \cite{gao2023extremal}.

\begin{claim} \label{claim:thick}
    If $|F_T| \ge 100 C^2 D n^2$, then $G$ contains $C^{\textrm{dia}}_6$.
\end{claim}

The proof of Claim~\ref{claim:thick} uses the following well-known bound for the asymmetric Zarankiewicz problem.
\begin{theorem}[K\H{o}v\'{a}ri, S\'{o}s and Tur\'{a}n~\cite{kHovari1954problem}] \label{theorem:zarankiewicz-c4}
    A bipartite graph with parts of size $m$ and $n$ and no four-cycle has at most $m \sqrt{n} + n$ edges.
\end{theorem}
\begin{proof}[Proof of Claim~\ref{claim:thick}]
    Suppose that $|F_T| \ge 100 C^2 D n^2.$ By averaging there is an edge $uv$ in $G$ which is contained in at least $|F_T| / (2Cn^{3/2}) \ge 50 CD n^{1/2}$ thick four-cycles and without loss of generality, at least $25 CD n^{1/2}$ of those four-cycles are of the form $uvxy$ such that $d_G(u, x) \ge 10 \sqrt{CD} n^{1/4}.$ Let $\{x_1, \dots, x_t\}$ denote the set of all vertices $x \in N(v)$ such that $d_G(u, x) \ge 10 \sqrt{CD} n^{1/4}.$ Consider the bipartite graph $H = G[ N(u), \{x_1, \dots, x_t\}].$ Every thick four-cycle of the form $uvxy$ with $d_G(u, x) \ge 10 \sqrt{CD} n^{1/4}$ corresponds to an edge $xy$ in $H$. Hence we have $e(H) \ge 25 CD n^{1/2}$ and by the definition of the set $\{x_1, \dots, x_t\}$, we have $e(H) \ge t \cdot 10 \sqrt{CD} n^{1/4}.$ So,
    \[ e(H) > 10 CD n^{1/2} + 5t \sqrt{CD} n^{1/4} > \Delta(G) + t \sqrt{\Delta(G)} \ge |N(u)| + \sqrt{|N(u)|}t. \]
    Thus, by Theorem~\ref{theorem:zarankiewicz-c4}, $H$ contains a four-cycle. The four vertices of this four-cycle along with $\{u, v\}$ form a $K_{3,3}.$ Finally, note that $K_{3,3}$ is isomorphic to $C^{\textrm{dia}}_6$, the $6$-cycle with all diagonals, finishing the proof of the claim.    
\end{proof}

From now on we assume $|F_T| < 100 C^2 D n^2$ since, otherwise, we are done by Claim~\ref{claim:thick}. We shall use the following well-known supersaturation theorem for four-cycles.
\begin{theorem}[Erd\H{o}s and Simonovits \cite{erdHos1983supersaturated}]  \label{thm:supersaturation}
    Let $C \ge 10$ and let $G$ be an $n$-vertex graph with $Cn^{3/2}$ edges. Then, $G$ contains at least $C^4 n^2 / 2$ four-cycles. 
\end{theorem}

\subsection{Finding an almost-spanning expander $\Gamma$ in the $C_4$-graph of $G$}
\label{subsec:mainlemma}

In this subsection, we prove our main lemma, Lemma~\ref{lem:almost_spanning_expander}, which shows that we can turn $\Gamma_0$ into an excellent (robust) expander $\Gamma$ by removing very few edges and vertices from $\Gamma_0$. As explained in the proof overview, the following notion of robust expansion of vertex sets in the $C_4$-graph is crucial to our proof. Related notions of robust expansion (but for \emph{sublinear} expansion) were recently used in \cite{haslegrave2022extremal, bucic2022towards, sudakov-tomon, letzter2023hypergraphs}.

\begin{defn}[Robust expansion]\label{def:robust-expansion}
    Let $\Gamma$ be an arbitrary subgraph of $\Gamma_0$. We say that a set $E_b \subseteq V(\Gamma)$ is \emph{$\delta$-robustly-expanding} in $\Gamma$ if for any set $F \subseteq E(\Gamma)$ with $|F| \le |E_b| \cdot \delta C^3 \sqrt{n},$ we have $|N_{\Gamma - F}(E_b)| \ge \delta |E_b|.$ If all sets of size at most $0.98 |V(\Gamma)|$ are $\delta$-robustly expanding in $\Gamma,$ then we say $\Gamma$ is a $\delta$-robust-expander.
\end{defn}

Denote $N_0 \coloneqq |V(\Gamma_0)| = C n^{3/2} = e(G).$ By Theorem~\ref{thm:supersaturation}, we have $e(\Gamma_0) \ge C^4 n^2 / 2.$ The following claim shows that all large enough vertex-sets of size up to $0.99N_0$ in $\Gamma_0$ are robustly expanding in the sense defined above.

\begin{claim} \label{cl:large-sets-robustly-expand}
    Any set $E_b \subseteq V(\Gamma_0)$ with $\eps N_0 \le |E_b| \le 0.99N_0$ is $2\delta$-robustly-expanding in $\Gamma_0$.
\end{claim}

\begin{proof}[Proof of Claim~\ref{cl:large-sets-robustly-expand}]
    Let $E_b \subseteq V(\Gamma_0)$ be a set with $\eps N_0 \le |E_b| \le 0.99N_0$ and suppose for the sake of contradiction that there is a set $F \subseteq E(\Gamma_0)$ such that $|F| \le |E_b| \cdot 2\delta C^3 \sqrt{n}$ and $|N_{\Gamma_0 - F}(E_b)| < 2\delta |E_b|.$ Let $E_p = N_{\Gamma_0 - F}(E_b)$ and $E_r = V(\Gamma_0) \setminus (E_b \cup E_p).$ Note that $|E_p| < 2 \delta |E_b|$. We shall call the edges in $E_b, E_p, E_r \subseteq E(G)$ blue, purple and red, respectively. Note that these are edges in $G$ and thus vertices in $\Gamma_0.$ For each $v \in V(G)$, let $d_b(v), d_p(v), d_r(v)$ denote the number of blue, purple and red edges incident to $v$, respectively. We will partition the vertex set of $G$ into four sets $P, R, B, U$ such that for any $v \in V(G)$ we have:
    \begin{align*}
        v \in
        \begin{cases}
            P,& \text{if } d_p(v) \ge d(v) / 2,\\
            R,& \text{if } d_p(v) < d(v) / 2 \text{ and } d_b(v) < \eta d(v),\\
            B,& \text{if } d_p(v) < d(v) / 2 \text{ and } d_r(v) < \eta d(v),\\
            U,& \text{if } d_p(v) < d(v) / 2 \text{ and } d_b(v), d_r(v) \ge \eta d(v).
        \end{cases}
    \end{align*}
    Note that since $\eta < 1/4,$ $P \cup R \cup B \cup U$ is a partition of $V(G).$ First, we will show that there are few edges of $G$ incident to $U$. To this end, let $e_U = \sum_{v \in U} d(v).$ Note that for any $v \in U,$ the number of 2-paths $xvy$ such that $xv \in E_r, vy \in E_b$ is at least $d_r(v) d_b(v) \ge \eta d(v) (1/2 - \eta) d(v) \ge \frac{\eta}{4} d(v)^2.$ By Jensen's inequality, the total number of 2-paths $xvy$ with $v \in U,$ $xv \in E_r, vy \in E_b$ (i.e., $xv$ is red and $vy$ is blue) is at least 
    \[ \frac{\eta}{4} |U| \left(\frac{e_U}{|U|}\right)^2 = \frac{\eta}{4} \frac{e_U^2}{|U|} \ge \frac{\eta}{4}\frac{e_U^2}{n}. \]
    Hence, on average, for a pair $x, y$ there are at least $\frac{\eta}{4} \frac{e_U^2}{n^3}$ choices for $v \in U$ such that $xv \in E_r, vy \in E_b$. If $e_U \ge \frac{20}{\eta^{1/2}} n^{3/2} $, then $\frac{\eta}{4} \frac{e_U^2}{n^3} \ge 10$, so 
    by Jensen's inequality, there are at least
    \[ \binom{n}{2} \binom{\frac{\eta}{4} \frac{e_U^2}{n^3}}{2} \ge \frac{\eta^2}{100} \frac{e_U^4}{n^4} \]
    four-cycles in $G$ of the form $xuyv$ where $xu, xv$ are red and $uy, vy$ are blue. Let us call such four-cycles \emph{red-red-blue-blue} four-cycles. Note that any red-red-blue-blue four-cycle in $G$ yields an edge between $E_r$ and $E_b$ in $\Gamma_0.$ By definition, $\Gamma_0-F$ has no edges between $E_b$ and $E_r,$ therefore all of these red-red-blue-blue four-cycles in $G$ correspond to edges in $F.$ Therefore, we obtain
    \[ \frac{\eta^2}{100} \frac{e_U^4}{n^4} \le |F| \le |E_b| \cdot 2\delta C^3 \sqrt{n} \le 2\delta C^4 n^2, \]
    which implies $e_U  \le (2\delta)^{1/5} Cn^{3/2}$, where we used $\delta \ll \eta.$ Thus, using that $\frac{1}{C} \ll \delta, \eta,$ we have $$e_U = \sum_{v \in U} d(v) \le \max\{(2\delta)^{1/5} Cn^{3/2}, \frac{20}{\eta^{1/2}} n^{3/2} \} = (2\delta)^{1/5} Cn^{3/2}.$$

    Furthermore, the number of edges incident to $P$ is at most
    \[ e_P \coloneqq \sum_{v \in P} d(v) \le 2 \sum_{v \in P} d_p(v) \le 4 |E_p| \le 8 \delta |E_b|. \]

    Finally, let us bound the number of edges between the vertex sets $B$ and $R$. There are at most $2\eta e(G)$ blue edges between $B$ and $R$. Indeed, otherwise, there is a vertex $v \in R$ incident to at least $\eta d(v)$ blue edges, a contradiction. Analogously, there are at most $2\eta e(G)$ red edges between $B$ and $R$. Trivially, there are at most $|E_p|$ purple edges between $B$ and $R$. Thus, $e_G(B, R) \le 4\eta e(G) + |E_p| \le 4\eta e(G) + 2 \delta |E_b|.$

    Since $e_G(B, \overline{B}) \le e_U + e_P + e_G(B, R)$, putting all of the above bounds together, we have
    \[ e_G(B, \overline{B}) \le (2\delta)^{1/5} Cn^{3/2} + 8 \delta |E_b| + (4 \eta e(G) + 2 \delta |E_b|) \le 5 \eta e(G), \] 
    where we used that $\delta \ll \eta$ and $e(G) = Cn^{3/2}.$ Hence, since $G$ satisfies \ref{prop:edge-expansion} of Lemma~\ref{lem:expandersubgraph}, we have either $|B| \le \frac{1000 n}{D}$ or $|\overline{B}| \le \frac{1000 n}{D}$ or 
    \[ \min\{|B|, |\overline{B}|\} \le \frac{16 e_G(B,\overline{B})}{C n^{1/2}} \le 80 \eta n. \] 

    Since $D^{-1} \ll \eta,$ this implies $|B| \le 80 \eta n$ or $|B| \ge (1 - 80\eta)n.$ Suppose first $|B| \le 80 \eta n.$ Since $G$ satisfies~\ref{prop:no-dense-set}  of Lemma~\ref{lem:expandersubgraph}, we have $e_G(B) \le 2C|B|^{3/2}$. Therefore, we have
    \begin{align*}
        |E_b| &\le e_G(B) + \sum_{v \in R} d_b(v) + \sum_{v \in P} d(v) + \sum_{v \in U} d(v)\\
        &\le 2C |B|^{3/2} + 2\eta e(G) + 8 \delta |E_b| + (2\delta)^{1/5} Cn^{3/2}\\
        &\le (2(80\eta)^{3/2} + 2 \eta + 8\delta + (2\delta)^{1/5}) Cn^{3/2} < \eps N_0,
    \end{align*}
    where we used that $\eps \gg \eta, \delta.$ This contradicts our assumption that $|E_b| \ge \eps N_0.$

    Therefore, we have $|B| \ge (1 - 80 \eta) n$ and thus $|R| \le 80\eta n.$ Note that since $G$ satisfies \ref{prop:no-dense-set}  of Lemma~\ref{lem:expandersubgraph}, we have $e_G(R) \le 2C|R|^{3/2} \le 2 C (80\eta n)^{3/2}$. Thus, we can bound $|E_r|$ as follows.
    \begin{align*}
        |E_r| &\le e_G(R) + \sum_{v \in B} d_r(v) + \sum_{v \in P} d(v) + \sum_{v \in U} d(v) \\
        &\le 2C (80\eta n)^{3/2} + 2\eta e(G) + 8\delta e(G) + (2\delta)^{1/5} Cn^{3/2} \\
        &\le (2(80\eta)^{3/2} + 2 \eta + 8 \delta + (2\delta)^{1/5}) e(G) \le  3\eta e(G),
    \end{align*}
    where we used $1 \gg \eta \gg \delta.$ Thus,
    \[ |E_b| + |E_p| + |E_r| \le 0.99e(G) + 2\delta e(G) + 3 \eta e(G) < e(G), \]
    a contradiction. This proves the claim.    
\end{proof}

Now our aim is to show that $\Gamma_0$ contains an almost-spanning robust expander $\Gamma$. To that end, let $\Gamma' \coloneqq \Gamma_0 - F_T,$ where $F_T$ is the set of edges of $\Gamma_0$ corresponding to the thick four-cycles of $G$. Let $E_0$ be a maximal set of size at most $0.99N_0$ which is not $\delta$-robustly-expanding in $\Gamma'$. Then there exists a set $F_0$ of edges in $\Gamma'$ such that $|N_{\Gamma' - F_0}(E_0)| < \delta |E_0|$ and $|F_0| \le |E_0| \cdot \delta C^3 \sqrt{n}$. Let $\Gamma \coloneqq \Gamma' \setminus E_0$ and $N \coloneqq |V(\Gamma)|.$ In the next lemma, we show that $\Gamma$ is a robust expander containing all but at most $\varepsilon$-proportion of the vertices of $\Gamma_0$.

\begin{lemma}[Main lemma]
\label{lem:almost_spanning_expander}
    We have $|V(\Gamma)| = N \ge (1 - \eps)N_0$ and moreover, $\Gamma$ is a $\delta$-robust-expander.
\end{lemma}
\begin{proof}
    Suppose for a contradiction that $N < (1-\eps)N_0,$ i.e., $|E_0| > \eps N_0$ (since $N_0 - |E_0| = N$). Then, $N_{\Gamma_0 - (F_0 \cup F_T)}(E_0) < \delta |E_0|.$ However, by Claim~\ref{claim:thick}, $|F_0 \cup F_T| \le \delta C^3 \sqrt{n} |E_0| + 100C^2 D n^{2} \le 2 \delta C^3 \sqrt{n} |E_0|,$ where we also used that $1/C \ll 1/D \ll \delta, \eps$ and $N_0 = C n^{3/2}.$ This means that $E_0$ is not $2\delta$-robustly-expanding in $\Gamma_0$ and $\eps N_0  < |E_0| \le 0.99N_0,$ contradicting Claim~\ref{cl:large-sets-robustly-expand}. Thus $|E_0| \le \eps N_0$, so $N \ge (1 - \eps)N_0$, proving the first part of the lemma.
    
    Now suppose for a contradiction that there is a non-empty set $E_b$ of size at most $0.98N$ which is not $\delta$-robustly-expanding in $\Gamma$. Then there is a set $F_b$ of edges in $\Gamma$ such that $|N_{\Gamma - F_b}(E_b)| = |N_{\Gamma' \setminus E_0 - F_b}(E_b)| < \delta |E_b|$ and $|F_b| \le \delta C^3 \sqrt{n} |E_b|$. Let $E = E_0 \cup E_b$ and $F = F_0 \cup F_b.$ We have $|E| = |E_0| + |E_b| \le  \eps N_0 + 0.98N \le 0.99N_0,$ $|F| \le |F_0| + |F_b| \le \delta C^3 \sqrt{n} |E_0| + \delta C^3 \sqrt{n} |E_b| = \delta C^3 \sqrt{n} |E|$ and $|N_{\Gamma' - F}(E)| \le |N_{\Gamma' - F_0}(E_0)| + |N_{\Gamma' \setminus E_0 - F_b}(E_b)| < \delta |E_0| + \delta |E_b| = \delta |E|.$ Thus, $E$ is not $\delta$-robustly-expanding in $\Gamma'$, $|E| \le 0.99N_0,$ and $|E| = |E_0| + |E_b| > |E_0|$ contradicting the maximality of $E_0.$ This proves the lemma.
\end{proof}

\subsection{Finding an odd path and an even path connecting most pairs of vertices in $\Gamma$}

In this subsection, our main goal is to show that most pairs of vertices in $\Gamma$ are connected by an odd path as well as an even path avoiding a small set of vertices in $G$ (see Lemma~\ref{lem:even-odd}). We do this by exploiting the fact that $\Gamma$ is an almost-spanning expander of the $C_4$-graph of $G$.

Let $\phi \colon V(\Gamma) \rightarrow E(G)$ be the bijection mapping vertices of  $\Gamma$ to the corresponding edges in $G$. Note that, by Lemma~\ref{lem:almost_spanning_expander}, we have $N = |V(\Gamma)| \ge (1 - \varepsilon) N_0 \ge \frac{C}{2} n^{3/2}$, and $\Gamma$ is a $\delta$-robust-expander. 

The next claim shows that in order to find a cycle with all diagonals in $G$, it suffices to find an odd cycle in $\Gamma$ satisfying a certain property made precise in the following definition.

\begin{defn}[Proper path/cycle]
    A path or a cycle $x_1, \dots, x_\ell$ in $\Gamma$ is called \emph{proper} if $\phi(x_i) \cap \phi(x_j) = \emptyset,$ for all $1 \le i < j \le \ell$, i.e., the edges of $G$ corresponding to $x_i$'s are disjoint.
\end{defn}

\begin{claim} \label{claim:cycle}
    A proper odd cycle in $\Gamma$ corresponds to a cycle with all diagonals in $G$.
\end{claim}
\begin{proof}
    Let $x_1, \dots, x_\ell$ be a proper odd cycle in $\Gamma,$ and denote $\phi(x_i) = u_iv_i \in E(G)$ for any $i \in [\ell],$ where the vertices $u_1, \dots, u_\ell$ are in the same part of the bipartition of $G$. Since $x_1, \dots, x_\ell$ is a cycle, it follows that $u_iv_{i+1}, u_{i+1}v_i \in E(G)$ for all $i \in [\ell],$ where we denote $v_{\ell+1} = v_1$ and $u_{\ell+1} = u_1$.

    For $i \in [2\ell],$ denote
    \[ w_i =
    \begin{cases}
        u_i, &\text{if } i \le \ell \text{ and } i \text{ is odd};\\
        v_i, &\text{if } i \le \ell \text{ and } i \text{ is even};\\
        v_{i-\ell}, &\text{if } i > \ell \text{ and } i \text{ is even};\\
        u_{i-\ell}, &\text{if } i > \ell \text{ and } i \text{ is odd}.
    \end{cases}\]
    Because $x_1, \dots, x_\ell$ is a proper cycle, all the vertices $w_1, \dots, w_{2\ell}$ are distinct. It is straightforward to check that $w_iw_{i+1} \in E(G)$ for all $i \in [2\ell]$ (with $w_{2\ell+1} = w_1$) and $w_iw_{i+\ell} \in E(G)$ for all $i \in [\ell].$ In other words, the vertices $w_1, \dots, w_{2\ell}$ form a copy of $C^{\textrm{dia}}_{2\ell}$ in $G$. This proves the claim.
\end{proof}

Slightly abusing notation, for a path $P = x_1, \dots, x_{\ell}$ in $\Gamma,$ we denote $\phi(P) = \bigcup_{i=1}^{\ell} \phi(x_i)$ i.e., 
$\phi(P)$ is the set of vertices in $G$ which are contained in the edges of $G$ corresponding to the vertices of $P$ in $\Gamma$. We say that a vertex $v \in V(G)$ appears on $P$ if $v \in \phi(P).$
The following notion of a nice fan and Lemma~\ref{lem:fan} are inspired by the ideas of Letzter~\cite{letzter2023hypergraphs} (see also \cite{jiang2021rainbow, bip_smooth_Jiang}).

\begin{defn}[Fan]
    Given $x \in V(\Gamma)$, a \emph{fan} $\calP$ rooted at $x$ is a collection  $\{Q(y)\}_{y \in Y}$  of proper paths in $\Gamma$, where $Y \subseteq V(\Gamma)$ and for each $y \in Y$, $Q(y)$ is a proper path starting at $x$ and ending at $y$. The \emph{size} of a fan is the number of paths in it. A fan $\calP$ is \emph{nice} if every $v \in V(G) \setminus \phi(x)$ appears on at most $t \coloneqq C^{1.1} n^{5/4} \log n$ paths in $\calP$.
\end{defn}

The following lemma shows that by starting at any given vertex in $\Gamma$, we can reach almost all other vertices of $\Gamma$ via proper paths that are not too long while also avoiding a small set $B$ of vertices in $G$. Our notion of robust expansion (guaranteeing vertex expansion in $\Gamma$ even after the removal of any relatively small set of edges) is crucial for the proof of the lemma.

\begin{lemma} \label{lem:fan}
    Let $x_0 \in V(\Gamma)$ be arbitrary and let $B \subseteq V(G)$ be a set of at most $n^{1/4}$ vertices in $G$ such that $\phi(x_0) \cap B = \emptyset.$ Then, there is a nice fan $\calP$ rooted at $x_0$ of size at least $0.98 N$ such that every $P \in \calP$ is a proper path of length at most $\frac{2 \log N}{\delta}$ and  $\phi(P) \cap B = \emptyset$.
\end{lemma}
\begin{proof}
    Let $\calP$ be a nice fan rooted at $x_0$ of maximum size such that every path in $\calP$ is a proper path of length at most $\frac{2 \log N}{\delta} \le \frac{100 \log n}{\delta}$ and, for the sake of contradiction, assume that $|\calP| < 0.98 N.$ Let $T$ be the set of vertices in $G$ appearing on precisely $t$ paths in $\calP$ and note that $|T| \le \frac{200 N \log n}{\delta t} < n^{1/4},$ where we used $\delta \gg \frac{1}{C}.$ Denote $B' = T \cup B,$ so $|B'| \le 2n^{1/4}.$
    We recursively define a sequence of subsets $U_0 \subseteq U_1 \subseteq U_2 \ldots $ of $V(\Gamma)$ as follows. Let $U_0 \coloneqq \{x_0\},$ and suppose we have already defined $U_0, U_1, \ldots U_i$ with the property that for every $x \in U_i$, there is a proper path $P(x)$ in $\Gamma$ from $x_0$ to $x$ of length at most $i$ such that $\phi(P(x)) \cap B' = \emptyset$ and moreover, suppose $|U_{j}| \ge (1 + \delta) |U_{j-1}|$ for every $1 \le j \le i$.  If $|U_i| < 0.98N,$ we define $U_{i+1}$ as follows. Otherwise, we stop the procedure.

    Let $U_{i+1} \coloneqq  U_i \bigcup \{ y \in V(\Gamma) \setminus U_i \mid \exists x \in U_i, xy \in E(\Gamma), \phi(y) \cap (\phi(P(x)) \cup B') = \emptyset \}$. Note that for every vertex $y \in U_{i+1}$, there exists a proper path $P(y)$ in $\Gamma$ from $x_0$ to $y$ of length at most $i+1$ such that $\phi(P(y)) \cap B' = \emptyset$. 

    We claim that $|U_{i+1}| \ge (1 + \delta) |U_i|.$ Indeed, consider an arbitrary $x \in U_i$, and an arbitrary $b \in \phi(P(x)) \cup B'$. Let $\phi(x) = \{x_1, x_2\} \in E(G)$. Without loss of generality, assume that $x_1$ and $b$ are on the same side of the bipartition of $G$. Then, since no edge in $\Gamma$ corresponds to a thick four-cycle in $G$, there are at most $10 \sqrt{CD} n^{1/4}$ edges in $\Gamma$ of the form $xy$ where $b \in y.$ Let $F_i \coloneqq \{xy \in E(\Gamma) \mid x \in U_i, \phi(y) \cap (\phi(P(x)) \cup B') \not = \emptyset \}$. 
For any $x \in U_i$, since $P(x)$ has length at most $i$, we have $|\phi(P(x))| \le 2i$. Moreover, since $|U_{j}| \ge (1 + \delta) |U_{j-1}|$ for every $1 \le j \le i$, we have $N \ge |U_{i}| \ge (1 + \delta)^{i} \ge e^{\frac{\delta i}{2}},$ implying that $i \le \frac{2 \log N}{\delta} \le \frac{5 \log n}{\delta}$. Hence, by the above discussion, we have $$|F_i| \le |U_i| \cdot (2i + |B'|) \cdot 10 \sqrt{CD} n^{1/4} \le |U_i|\left (\frac{10 \log n}{\delta} + 2 n^{1/4}\right) 10 \sqrt{CD} n^{1/4} \le |U_i| \cdot 30 \sqrt{CD} n^{1/2} < |U_i| \cdot \delta C^3 n^{1/2},$$ where we used $\frac{1}{C} \ll \delta, \frac{1}{D}$. Since $\Gamma$ is $\delta$-robustly-expanding, this implies that $|N_{\Gamma - F_i}(U_i)| \ge \delta |U_i|$. Moreover, crucially, notice that $N_{\Gamma - F_i}(U_i) = \{ y \in V(\Gamma) \setminus U_i \mid \exists x \in U_i, xy \in E(\Gamma), \phi(y) \cap (\phi(P(x)) \cup B') = \emptyset \}$. Thus, $|U_{i+1}| \ge |U_i| + \delta |U_i|  = (1+ \delta) |U_i|,$ as claimed.

    Suppose the above procedure stops after $\ell$ steps having defined the sets $U_0 \subseteq U_1, \ldots \subseteq U_{\ell} \subseteq V(\Gamma)$. Then, again since $N \ge |U_{\ell}| \ge (1 + \delta)^{\ell} \ge e^{\frac{\delta \ell}{2}},$ we have $\ell \le \frac{2 \log N}{\delta}$ and since the procedure has stopped, we must have $|U_\ell| \ge 0.98N.$ Hence, there is a vertex $y \in U_{\ell}$ such that there is no path in $\calP$ from $x_0$ to $y$, but since $y \in U_{\ell}$, there is a proper path $P(y)$ in $\Gamma$ from $x_0$ to $y$ of length at most $\frac{2 \log N}{\delta}$ such that $\phi(P(y)) \cap B' = \emptyset$. However, then we could add the path $P(y)$ to $\calP,$ while still guaranteeing that $\calP \cup \{P(y)\}$ is a nice fan (as $\phi(P(y)) \cap T = \emptyset$), and ensuring that $\phi(P(y)) \cap B = \emptyset$. This contradicts the maximality of the fan $\calP$ and completes the proof of the lemma.
\end{proof}

The fact that $\Gamma$ is an almost-spanning expander of the $C_4$-graph of $G$ is crucial to the following claim. 

\begin{claim} \label{claim:c4-s}
    For any set $U \subseteq V(\Gamma)$ with $|U| \ge N / 100,$ it holds that $e(\Gamma[U]) \ge 10^{-11} C^4 n^2.$
\end{claim}

\begin{proof}
    Let $U \subseteq V(\Gamma)$ be a subset of vertices such that $|U| \ge N/100 \ge C n^{3/2}/200$. Note that $U$ corresponds to a set of at least $C n^{3/2} / 200$ edges in $G$. Using that $C$ is large enough, by Theorem~\ref{thm:supersaturation}, these edges induce at least $10^{-10} C^4 n^2$ four-cycles. Since $|F_T| \le 100C^2 Dn^2$ by Claim~\ref{claim:thick}, and $1/C \ll 1/D$, at least $10^{-10} C^4 n^2 - 100C^2 Dn^2 \ge 10^{-11} C^4 n^2$ of these four-cycles correspond to edges in $\Gamma[U],$ proving the claim.
\end{proof}

Using the claim above and Lemma~\ref{lem:fan}, we prove the following lemma which shows that by starting at any vertex in our almost-spanning expander $\Gamma$, we can reach almost all other vertices of $\Gamma$ via a short proper path of odd length as well as a short proper path of even length (simultaneously) while also avoiding a small set of vertices in~$G$.

\begin{lemma} \label{lem:even-odd}
     Let $x_0 \in V(\Gamma)$ be arbitrary and let $B \subseteq V(G)$ be a set of at most $n^{1/4}$ vertices in $G$ such that $\phi(x_0) \cap B = \emptyset.$ Then, there is a set $Y \subseteq V(\Gamma)$ of size $|Y| \ge 0.94N$ such that for every $y \in Y,$ there exists a proper path $P_o(y)$ of odd length in $\Gamma$ and a proper path $P_e(y)$ of even length in $\Gamma$, both of which start at $x_0$ and end at $y$ and have length at most $\frac{2 \log N }{\delta}  + 1$, and $(\phi(P_o(y)) \cup \phi(P_e(y))) \cap B = \emptyset$.
\end{lemma}
\begin{proof}
    By Lemma~\ref{lem:fan}, there is a nice fan $\calP$ rooted at $x_0$ of size at least $0.98N$ such that every $P \in \calP$ is a proper path of length at most $\frac{2 \log N}{\delta}$ and $\phi(P) \cap B = \emptyset$. Let $U \subseteq V(\Gamma)$ be the set of endpoints of the paths in $\calP$ and for each $y \in U,$ denote by $P(y)$ the path in $\calP$ from $x_0$ to $y$. Let $U_0$ and $U_1$ be the set of vertices $y \in U$ for which the length of the path $P(y) \in \calP$ is even and odd respectively. Then $|U| = |U_0| + |U_1| \ge 0.98 N$.
    
    Suppose $|U_0| \ge N/50$. Let $\Gamma'$ be a graph obtained from $\Gamma[U_0]$ by repeatedly removing vertices of degree at most $C^{2.5} n^{1/2}$. Note that $|V(\Gamma')| \ge |U_0| - N / 100.$ Indeed, otherwise, in $\Gamma$ there is a vertex set of size $N / 100$ spanning at most $N / 100 \cdot C^{2.5} n^{1/2} = C^{3.5} n^2 / 100 < 10^{-11} C^4 n^2$ edges (since $C$ is large enough) contradicting Claim~\ref{claim:c4-s}. Thus, in particular, $\Gamma'$ is non-empty and $\delta(\Gamma') \ge C^{2.5} n^{1/2}$.
    
    Note that for any edge $yz \in E(\Gamma')$ with $\phi(z) \cap (\phi(P(y)) \cup B) = \emptyset$, we can extend the path $P(y)$ from $x_0$ to $y$ using the edge $yz$ to obtain a proper path $P_o(z)$ in $\Gamma$ of odd length at most $\frac{2\log N}{\delta} + 1$ from $x_0$ to $z$ such that $\phi(P_o(z)) \cap B = \emptyset$.
 
    Let us show that such a path $P_o(z)$ of odd length can be obtained for all but at most $N / 100$ vertices $z \in V(\Gamma').$ To that end, we bound the number of edges $yz \in E(\Gamma')$ such that $\phi(z) \cap (\phi(P(y)) \cup B) \neq \emptyset$ as follows. 
    
    For each $y \in V(\Gamma')$ and $v \in \phi(P(y)) \cup B,$ there are at most $10 \sqrt{CD} n^{1/4}$ possible vertices $z \in V(\Gamma')$ such that $v \in \phi(z)$ and $yz \in E(\Gamma').$ Indeed, otherwise the vertex $u \in \phi(y)$ which is in the same part of $G$ as $v$ satisfies $d_G(u, v) > 10 \sqrt{CD} n^{1/4},$ so any such edge $yz \in E(\Gamma')$ would correspond to a thick four-cycle in $\Gamma_0$ and hence would not be in $\Gamma$, a contradiction. Since for any $y \in V(\Gamma')$, the path $P(y) \in \calP$ has length at most $\frac{2 \log N }{\delta}$ (so, $\phi(P(y)) \le \frac{4 \log N }{\delta}$) we conclude that the number of edges $yz \in E(\Gamma')$ such that $\phi(z) \cap (\phi(P(y)) \cup B) \neq \emptyset$ is at most $N \cdot ( \frac{4\log N}{\delta} + n^{1/4}) \cdot 10 \sqrt{CD} n^{1/4} < C^{1.6} n^2,$ where we used $\frac{1}{C} \ll \frac{1}{D}$.
    
    On the other hand, for a vertex $z \in V(\Gamma'),$ if we cannot find a path $P_o(z)$ of odd length from $x_0$ to $z$ as described above, then $\phi(z) \cap (\phi(P(y) \cup B) \neq \emptyset$ for all $y$ with $yz \in E(\Gamma').$ Hence, the number of such vertices $z$ is at most $C^{1.6} n^2 / \delta(\Gamma') \le n^{3/2}/C^{0.9} < N / 100,$ as required. 
    
    In summary, if $|U_0| \ge N/50$, we obtain a proper odd path $P_o(z)$ of length at most $\frac{2 \log N}{\delta}  + 1$ from $x_0$ to $z$ with $\phi(P_o(z)) \cap B = \emptyset$ for all but at most $N / 50$ vertices $z \in U_0.$ If $|U_0| < N/50$, then this statement is also vacuously true. An analogous argument yields an even proper path $P_e(z)$ of length at most $\frac{2 \log N }{\delta} + 1$ from $x_0$ to $z$ with $\phi(P_e(z)) \cap B = \emptyset$ for all but at most $N / 50$ vertices $z \in U_1.$
    
    Now, for each $z \in U_0$, let $P_e(z) \coloneqq P(z) \in \calP$ and for each $z \in U_1$, let $P_o(z) \coloneqq P(z) \in \calP$. Hence, the desired paths $P_e(z), P_o(z)$ exist for at least $|U| - 2 (N/50) \ge 0.98N - 2(N/50) = 0.94N$ vertices $z \in V(\Gamma)$, as required. This completes the proof of the lemma.
\end{proof}

In the next subsection, we combine Lemma~\ref{lem:even-odd}, Lemma~\ref{lem:fan} and ideas in \cite{letzter2023hypergraphs} to construct a proper odd cycle in $\Gamma$, giving us the required cycle with all diagonals in $G$.

\subsection{Constructing a proper odd cycle in $\Gamma$ and finding a cycle with all diagonals}
\label{sec:closingoddcycle}
By Claim~\ref{claim:cycle}, in order to find a cycle with all diagonals in $G$, it suffices to find a proper cycle of odd length in $\Gamma$. To that end, we apply Lemma~\ref{lem:fan} with $B = \emptyset$ to obtain for every $x \in V(\Gamma),$ a nice fan $\calP(x) \coloneqq \{ P(x,y) \}_{y \in Y(x)}$  rooted at $x$ of size at least $0.98 N$, where $Y(x) \subseteq V(\Gamma)$ and for every $y \in Y(x)$, $P(x, y)$ is the proper path in $\calP(x)$ starting at $x$ and ending at $y$.

Now, for every $y \in V(\Gamma),$ let $F(y)$ be the set of vertices in $V(G) \setminus \phi(y)$ appearing on at least $t$ of the paths among $\{ P(x, y) \}_{x \in V(\Gamma)}$. Since $t = C^{1.1} n^{5/4} \log n,$ and for every $x \in V(\Gamma)$, the path $P(x, y)$ has length at most $\frac{2 \log N}{\delta} \le \frac{5 \log n}{\delta}$, we have
    \[ |F(y)| \le N \cdot \frac{10 \log n}{\delta} \cdot \frac{1}{t} \leq Cn^{3/2} \cdot \frac{10 \log n}{\delta} \cdot \frac{1}{C^{1.1} n^{5/4} \log n} \le
    \frac{10}{\delta C^{0.1}} n^{1/4} < n^{1/4}, \]
    where we used $\frac{1}{C} \ll \delta.$

    Thus, applying Lemma~\ref{lem:even-odd} with $F(y)$ playing the role of $B$, we obtain for every $y \in V(\Gamma),$ a set $S(y) \subseteq V(\Gamma)$ of size $|S(y)| \ge 0.94N$ such that for every $x \in S(y),$ there is a proper path $P_o(y, x)$  of odd length and a proper path $P_e(y, x)$ of even length in $\Gamma$, both of which start at $y$ and end at $x$ and are of length at most $\frac{2 \log N }{\delta} + 1$ such that $(\phi(P_o(y, x)) \cup \phi(P_e(y,x))) \cap F(y) = \emptyset$.

    \noindent We claim that if there is a pair of vertices $x, y \in V(\Gamma)$ satisfying the properties $\mathrm{(i)}$ and $\mathrm{(ii)}$ below, then there is a proper cycle of odd length in $\Gamma$, as desired.
    \begin{itemize}
    
        \item[(i)] \label{proof:manyz} there is a set $Z \subseteq V(\Gamma)$ of size $|Z| \ge 0.7 N$ such that for every $z \in Z$, we have proper paths $P(x, z) \in \calP(x)$, $P(z, y) \in \calP(z)$ and $\phi(P(x,z)) \cap \phi(P(z, y)) = \phi(z)$, 
        \item[(ii)] \label{proof:backtox} $x \in S(y)$.
    \end{itemize}
     
    Let us first prove the claim above. Indeed, since $\mathrm{(ii)}$ holds, there are proper paths $P_o(y, x),$ $P_e(y, x)$ from $y$ to $x$ as defined above. Since the paths are proper, note that in particular $\phi(x) \cap \phi(y) = \emptyset$.  Let $T \coloneqq \phi(P_o(y,x)) \cup \phi(P_e(y,x)) \setminus (\phi(x) \cup \phi(y))$ and note that $|T| \le 4 \left(2\frac{\log N}{\delta} + 1 \right) < 20 \frac{\log n}{\delta}.$ 
    Since $\calP(x)$ is a nice fan and $T \cap \phi(x) = \emptyset$, every vertex $v \in T$ appears on at most $t$ of the paths $\{P(x,z)\}_{z \in Z}$. Moreover, since 
    $T \cap (F(y) \cup \phi(y))  = \emptyset$, every vertex $v \in T$ appears on at most $t$ of the paths $\{P(z,y)\}_{z \in Z}$. Then it follows that there are at least $|Z| - 2|T| \cdot t \ge 0.7N - 2 \left(20 \frac{\log n}{\delta}\right)
    C^{1.1} n^{5/4} \log n \ge 0.6N$ vertices $z \in Z$ such that $(\phi(P(x,z)) \cup \phi(P(z, y))) \cap T = \emptyset$. Fix one such vertex $z \in Z.$ Then since $\phi(P(x,z)) \cap \phi(P(z, y)) = \phi(z)$ by $\mathrm{(i)}$, either $P(x, z)P(z, y)P_e(y,x)$ or $P(x, z)P(z, y)P_o(y,x)$ is a proper odd cycle in $\Gamma,$ proving the claim.

    Hence, to complete the proof, it suffices to find a pair of vertices $x, y \in V(\Gamma)$ satisfying the properties $\mathrm{(i)}$ and $\mathrm{(ii)}$.

To that end, let us consider the ordered triples $(x, z, y) \in V(\Gamma) \times V(\Gamma) \times V(\Gamma)$ such that $P(x, z) \in \calP(x)$,  $P(z, y) \in \calP(z)$ and $\phi(P(x,z)) \cap \phi(P(z, y)) = \phi(z)$, and call such triples $(x, z, y)$ \emph{admissible}. Let us give a lower bound on the number of admissible triples $(x, z, y)$. For a given $x \in V(\Gamma)$, the number of choices of $z \in V(\Gamma)$ such that $P(x, z) \in \calP(x)$ is at least $0.98 N$. Fixing a choice of $x, z \in V(\Gamma)$ with $P(x, z) \in \calP(x)$, every vertex $v \in \phi(P(x, z)) \setminus \phi(z)$ appears on at most $t$ of the paths in $\calP(z)$ because $\calP(z)$ is a nice fan. Hence, the number of choices for $y \in V(\Gamma)$ such that $P(z, y) \in \calP(z)$ and $\phi(P(z, y)) \cap (\phi(P(x,z)) \setminus \phi(z)) = \emptyset$ is at least $|\calP(z)| - |\phi(P(x, z))| t \ge  0.98 N - \frac{4 \log N}{\delta} \cdot C^{1.1} n^{5/4} \log n \ge 0.97 N$. Hence the number of admissible triples $(x, z, y)$ is at least $N \cdot (0.98 N) \cdot (0.97 N)$. 
If an ordered pair $(x, y) \in V(\Gamma) \times V(\Gamma)$ does not satisfy $\mathrm{(i)}$, then there are more than $N- 0.7N$ vertices $z$ such that  $(x, z, y)$ is not an admissible tuple. Thus the number of ordered pairs $(x, y) \in V(\Gamma) \times V(\Gamma)$ which do not satisfy $\mathrm{(i)}$ is at most $\frac{(1 - 0.98 \cdot 0.97)N^3}{(1-0.7)N} < \frac{N^2}{3}$. The number of ordered pairs $(x, y) \in V(\Gamma) \times V(\Gamma)$ such that $x \not\in S(y)$ (i.e., which do not satisfy $\mathrm{(ii)}$) is at most $N(N - 0.94N) = 0.06N^2.$ In total, the number of ordered pairs $(x, y) \in V(\Gamma) \times V(\Gamma)$ not satisfying either $\mathrm{(i)}$ or $\mathrm{(ii)}$ is at most $\frac{N^2}{3} + 0.06N^2 < N^2,$ so there is a  pair $(x, y)$ satisfying $\mathrm{(i)}$ or $\mathrm{(ii)}$, as required. This completes the proof of Theorem~\ref{thm:main}. 

\section{Concluding remarks} \label{sec:concluding}
In this paper we proved that any graph with at least $C n^{3/2}$ edges contains a cycle with all diagonals (where $C$ is a large enough constant). A key lemma in our proof is Lemma~\ref{lem:almost_spanning_expander}, which states that, roughly speaking, if $G$ is an expander, then its $C_4$-graph contains an almost-spanning robust expander. Here the $C_4$-graph of $G$ is the graph defined on the edge-set of $G$ in which two edges are adjacent if they are opposite edges of a four-cycle in $G$. We believe that similar arguments applied to different auxiliary graphs could be useful for other problems.

Let us mention a couple of problems related to our result. As noted in the introduction, if $\ell = 3$, then $C^{\textrm{dia}}_{2\ell}$ is isomorphic to the complete bipartite graph $K_{3,3}$, so $\ex(n, C^{\textrm{dia}}_{2\ell}) = \ex(n, K_{3,3}) = \Omega(n^{5/3})$ using Brown's well-known construction~\cite{brown1966graphs}. On the other hand, our proof shows that every $n$-vertex graph with $C n^{3/2}$ edges contains a copy of $C^{\textrm{dia}}_{2\ell}$ for some $\ell = O(\log n)$. This suggests the following question.

\begin{problem}
    Is it true that if $\ell \in \mathbb{N}$ is large enough, then $\ex(n, C^{\textrm{dia}}_{2\ell}) = O(n^{3/2})$?
\end{problem}

Next, we propose a possible generalization of our result. For a natural number $k$, consider an even cycle with every $k$--th diagonal and note that the graphs $C^{\textrm{dia}}_{2\ell}$ correspond to taking $k=1$. It is easy to see that if $k$ is even, such a graph cannot be bipartite. For odd $k,$ much like in Figure~\ref{fig:c10}, this graph can be viewed as an ``odd cycle of $2(k+1)$-cycles" where two consecutive cycles share exactly one edge. Hence, can we ask the following question. For any odd $k \ge 3,$ is there is a constant $C = C(k)$ such that every $n$-vertex graph with at least $C n^{1 + 1/(k+1)}$ edges contains a cycle with every $k$--\text{th} diagonal? We believe that the methods developed in this paper might be useful in resolving this problem. More precisely, one would require an analogue of Claim~\ref{cl:large-sets-robustly-expand} for the auxiliary graph $\Gamma_0$ defined on the edge-set of $G$ in which two edges of $G$ are adjacent in the auxiliary graph if they are opposite edges in a $2(k+1)$-cycle. Additionally, one would need a way to upper bound the number of $2(k+1)$-cycles containing a given edge and vertex. This would be analogous to Claim~\ref{claim:thick}.

\vspace{0.3cm}
\noindent
{\bf Acknowledgements.}
We are extremely grateful to Jun Gao, Oliver Janzer, Tao Jiang, Hong Liu and Zixiang Xu for valuable discussions.

\end{document}